\documentclass[11pt]{amsart}
\usepackage{geometry}                
\usepackage{graphicx}
\usepackage{amssymb}
\usepackage{epstopdf}
\usepackage{amssymb, amsmath, amscd, amsthm,
color, epsfig
}
\usepackage[all]{xy}          
\xyoption{dvips}              

\title[Context-free manifold calculus and the Fulton-MacPherson Operad]{Context-free manifold calculus and the Fulton-MacPherson Operad}
\author{Victor Turchin}
\address{Department of Mathematics\\ Kansas State University \\ Manhattan, KS 66506 \\ USA}
\email{turchin@ksu.edu}
\thanks{Supported  by NSF grant DMS 0968046}
\date{}                                           

\newcommand{\calB}{\mathcal{B}}

\newcommand{\calO}{\mathcal{O}}

\newcommand{\R}{{\mathbb R}}

\newcommand{\Emb}{\operatorname{Emb}}
\newcommand{\Ebar}{\overline{\Emb}}
\newcommand{\Imm}{\operatorname{Imm}}
\newcommand{\Top}{{\mathbb T}\operatorname{op}}

\newcommand{\hIbimod}{\operatorname{hIbimod}}

\newcommand{\ind}{\operatorname{ind}}

\usepackage{graphicx, psfrag,rotating}

\DeclareMathOperator*{\holim}{holim}

\DeclareMathOperator*{\hNat}{hNat}
\DeclareMathOperator*{\hRmod}{hRmod}

\newcommand{\Ebarmn}{{\overline{\mathrm{Emb}}}_c(\R^m,\R^n)}

\newcommand{\Rmod}{\operatorname{Rmod}}

\newcommand{\Assoc}{{\mathcal A}ssoc}

\newcommand{\Man}{{\mathbb M}{\mathrm{an}}}
\newcommand{\SSS}{{\mathbb S}}
\newcommand{\Maps}{{\mathrm{Maps}}}
\newcommand{\J}{{\mathcal J}}
\newcommand{\Edm}{{\mathcal E}(D^m)}
\newcommand{\F}{{\mathcal F}}
\newcommand{\GL}{{\mathrm{GL}}}
\newcommand{\Tot}{{\mathrm{Tot}}}

\newcommand{\calN}{{\mathcal N}}
\newcommand{\calK}{{\mathcal K}}

\newcommand\rth{\refstepcounter{equation}}
\newcommand\numb{\rth{\rm \theequation}}
\numberwithin{equation}{section}

\theoremstyle{plain}
\newtheorem{theorem}{Theorem}[section]
\newtheorem{proposition}[theorem]{Proposition}
\newtheorem{lemma}[theorem]{Lemma}

\theoremstyle{definition}
\newtheorem{definition}[theorem]{Definition}

\theoremstyle{remark}

\begin{document}

\sloppy

\maketitle

\begin{abstract}
The paper gives an explicit description of the Weiss embedding tower in terms of spaces of maps of truncated  modules over the framed Fulton-MacPherson operad.
\end{abstract}


\setcounter{section}{0}

\section*{Organization of the paper}
In Section~\ref{s1} we outline a general framework of {\it context free manifold calculus} and its connection to the framed discs operad. The details of this approach were completed by P.~Boavida de
 Brito and M.~Weiss in~\cite{BoavidaWeiss}. One of the main results of their work is Theorem~\ref{t_cont_free} that describes the Weiss Taylor tower of a context free topological presheaf on a manifold in terms of
maps of truncated right modules over the  framed discs operad. Section~\ref{s1} is given to emphasize the fact that for this description one can use the operad of framed discs with both its usual and discrete
topology. A discrete version of Theorem~\ref{t_cont_free} appeared earlier in a work of G.~Arone and the author~\cite{ArTur}. Section~\ref{s2} is where the main construction  is given. In Theorem~\ref{t_fult_mac}  we
replace the framed discs operad by the framed Fulton-MacPherson operad and describe  Weiss' approximations $T_k\Emb(M,N)$ to the space $\Emb(M,N)$ of embeddings of one manifold into another
in terms of maps of truncated right modules over the latter operad. The right modules in question  themselves are naturally obtained from the Axelrod-Singer-Fulton-MacPherson compactifications of
framed configuration spaces in manifolds $M$ and $N$. This description of the embedding tower resembles both the Goodwillie-Klein-Weiss construction~\cite{GKW}  and also Sinha's models~\cite{Sinha,Sinha-OKS}
for spaces of one dimensional knots. The proof of Theorem~\ref{t_fult_mac} is very straighforward and does
not rely on the Boavida-Weiss Theorem~\ref{t_cont_free}. Sections~\ref{s3} and~\ref{s4} do this job. Moreover our construction can be used to give an alternative proof of Theorem~\ref{t_cont_free} which is shown in Section~\ref{s5}. Section~\ref{s6} produces another application of our construction describing Weiss' tower for spaces of long embeddings in terms of maps of truncated infinitesimal bimodules over the Fulton-MacPherson operad. As a corollary we obtain that for $n>m+2$ the space $\Ebarmn$ is equivalent to the space of derived maps between the operads of little discs $\calB_m$ and $\calB_n$ in the category of infinitesimal bimodules over $\calB_m$.

\section{Context free manifold calculus and the operad of framed discs}\label{s1}
In~\cite{Weiss} M.~Weiss introduced the so called {\it manifold calculus of functors}. Given a smooth manifold $M$, denote by $\calO(M)$ the category of open subsets of $M$. For any isotopy invariant cofunctor $F\colon\calO(M)\to\Top$ in topological spaces, Weiss defines a {\it Taylor tower}
$$
\xymatrix{
&F\ar[dl] \ar[d] \ar[dr] \ar[drr]\\
T_0F&T_1F\ar[l]&T_2F\ar[l]&T_3F\ar[l]&\ldots\ar[l]
}
\eqno(\numb)\label{eq_tower}
$$
of {\it polynomial approximations} of $F$. It became clear a while ago that the manifold calculus of functors is deeply related to the operad of little discs. Below we outline one of the constructions that shows this connection.

Let $\Man_m$ denote the {\it category of smooth $m$-manifolds}, where the morphisms are codimension zero embeddings. This category is naturally enriched in topological spaces. We denote by $^\delta\Man_m$ its discretization. For any $m$-manifold $M$ one has an obvious forgetful functor
$$
I_M\colon\calO(M)\to{} ^\delta\Man_m.
$$

\begin{definition}\label{d_cont_free} A cofunctor $\widetilde{F}\colon\calO(M)\to\Top$ is {\it context-free} if $\widetilde{F}$  up to a natural equivalence factors through $^\delta\Man_m$. In other words $\widetilde{F}\simeq F\circ I_M$ for some cofunctor $F\colon^\delta\Man_m\to\Top$.
\end{definition}

In the sequel by a context-free cofunctor we will often understand the underlying cofunctor $F\colon^\delta\Man_m\to\Top$. Notice that this definition is slightly different and somewhat simpler than the one previously used, see~\cite[Definition~4.9]{ArTur}. But the idea is still the same --- a cofunctor is context-free if \lq\lq it does not depend" on where the open subsets are located. As an example, consider a non-trivial fibration $p\colon E\to M$, then the cofunctor $\Gamma(-,p)$ of continuous sections of $p$ is linear, but in general not context-free. The context-free cofunctors abound. The embedding and immersion cofunctors $\Emb(-,N)$, $\Imm(-,N)$ are context-free. These cofunctors assign to an open set $U\subset M$ the space of smooth embeddings, respectingly immersions, of $U$ in another smooth manifold $N$ of dimension $\geq m$. As a further generalization for any type of multisingularity $\SSS$ the spaces $\Maps_\SSS(M,N)$ of smooth maps  $M\to N$ that avoid $\SSS$ also define a context-free cofunctor
$$
\Maps_\SSS(-,N)\colon\Man_m\to\Top.
$$
All these cofunctors are in fact {\it continuous} in the sense that they are defined on the enriched category $\Man_m$. Given a context-free cofunctor it is natural to forget about the initial manifold $M$ and study the calculus of cofunctors with domain $^\delta \Man_m$. We will call such calculus {\it context-free manifold calculus}. This variation of manifold calculus is actually more similar to their brothers homotopy calculus~\cite{Goodwillie1,Goodwillie2,Goodwillie3} and orthogonal calculus~\cite{WeissOrth} since it deals with all manifolds similarly as the homotopy calculus deals with all topological spaces or spectra, and the orthogonal calculus deals with all vector spaces of finite dimension.

By an obvious analogy with~\cite{Weiss} define a Grothendieck topology $\J_k$ on $^\delta\Man_m$ in which $\{U_i\stackrel{f_i}{\hookrightarrow} V\}_{i\in I}$ is a $\J_k$-cover if and only if $\bigcup_{i\in I}f_i(U_i)^{\times k}=V^{\times k}$. In other words any configuration of $\leq k$ points in $V$ should appear in the image of at least one $U_i$.  For our purposes we will be using the following definition of polynomial functors. It is actually a non-trivial result of Weiss that the following is equivalent to a more usual definition that uses cubical diagrams~\cite{Weiss}.

\begin{definition}\label{d_polynomial}
An isotopy invariant presheaf $F\colon ^\delta\Man_n\to\Top$ is called {\it polynomial of degree $\leq k$} if it is a homotopy $\J_k$-sheaf.
\end{definition}

A reader unfamiliar with the notion of a homotopy sheaf may wait until Section~\ref{s3} where we explain what this property means.

Let $^\delta\calO_{\leq k}$ (respectively $\calO_{\leq k}$) be the full subcategory of $^\delta\Man_m$ (respectively $\Man_m$) whose objects are disjoint unions of $\leq k$ standard $m$-balls. Thus this category has only $k+1$ objects. Define $T_kF$ as the homotopy right Kan extension of $F$ from $^\delta\calO_{\leq k}$ to $^\delta\Man_m$:
$$
T_kF(M)=\holim\limits_{^\delta\calO_{\leq k}\downarrow M} F.
\eqno(\numb)\label{eq_Tk}
$$

For every $M$ one has a natural map $F(M)\to T_kF(M)$. Denote by $\eta_F$ the corresponding natural transformation.

\begin{proposition}\label{p_sheafification}
For any isotopy invariant presheaf $F\colon^\delta\Man_m\to\Top$ the natural transformation $\eta_F\colon F\to T_kF$ is a homotopy $\J_k$ sheafification  of $F$ in the sense that
\begin{itemize}
\item $T_kF$ is polynomial of degree $\leq k$;
\item in case $F$ is polynomial of degree $\leq k$ then $\eta_F$ is a natural equivalence.
\end{itemize}
\end{proposition}

\begin{proof}
One obviously has that $F$ is polynomial of degree $\leq k$ if and only if its restriction on $\calO(M)$ is so for all $M\in{}^\delta\Man_m$. Let $\calO_{\leq k}(M)$ denote the subcategory of open subsets of $M$ diffeomorphic to a disjoint union of $\leq k$ balls. One has a natural evaluation functor
$$
ev_M\colon ^\delta \calO_{\leq k} \downarrow M\to \calO_{\leq k}(M)
$$
that assigns to any embedding $f\colon U\hookrightarrow M$, $U\in {}^\delta\calO_{\leq k}$, its image $f(U)$. This cofunctor is homotopy right cofinal\footnote{Actually it is both right and left cofinal, but we care only about the right cofinality since we only need that $ev_M$ preserves homotopy limits and our functors are contravariant.} and therefore the induced map
$$
\holim_{\calO_{\leq k}(M)} F\to \holim_{^\delta\calO_{\leq k}\downarrow M} F
$$
is a weak equivalence. Notice that the first homotopy limit is exactly Weiss' formula for the $k$-th approximation. Thus the properties of $T_kF$ mentioned in the proposition follow  from the analogous properties of Weiss' approximations~\cite{Weiss}.
The cofinality of $ev_M$ is immediate from the fact that for any $V\in\calO_{\leq k}(M)$ the corresponding undercategory $V\downarrow ev_M$ has initial objects.
\end{proof}

Yet there is another way to describe $T_kF$.

\begin{lemma}\label{l_nat}
For any presheaf $F\colon ^\delta\Man_m\to\Top$ one has a natural equivalence
$$
\holim_{^\delta\calO_{\leq k}\downarrow M} F \simeq \hNat_{^\delta\calO_{\leq k}}\left(^\delta\Emb(\bullet,M),F(\bullet)\right).
\eqno(\numb)\label{eq_equiv}
$$
\end{lemma}

In the above $^\delta\Emb(\bullet,M)$ is a cofunctor that assigns to any $U\in{}^\delta\calO_{\leq k}$ the space of embeddings $\Emb(U,M)$ with discrete topology; $\hNat$ denotes the space of homotopy natural transformations betsween cofunctors on $^\delta\calO_{\leq k}$. This lemma is a particular case of~\cite[Lemma~3.7]{Arone}. The idea of the proof is that both spaces can be described as a totalization of certain cosimplicial spaces, and moreover the second corresponding cosimplicial space is obtained by an edgewise subdivision of the first one. Thus the equivalence~\eqref{eq_equiv} can be viewed as a natural homeomorphism. The latter description of $T_kF$ has a nice interpretation from the point of view of the theory of operads. Notice that the category $\Man_m$, respectively $^\delta\Man_m$, is symmetric monoidal where the monoidal structure is given by disjoint union, and unit is the emptyset. Let $\Edm$, respectively $^\delta\Edm$, denote the operad of endomorphisms of the unit disc $D^m$ in $\Man_m$, respectively $^\delta\Man_m$. It is obvious that $\Edm$ is equivalent to the operad of framed discs, and $^\delta\Edm$ is simply the discretization of $\Edm$. Next notice that a cofunctor $G\colon\calO_{\leq k}\to\Top$, respectively $G\colon^\delta\calO_{\leq k}\to\Top$, is exactly the same thing as a $k$-truncated right module over $\Edm$, respectively $^\delta\Edm$. Indeed, given such functor define a sequence of $k+1$ spaces $G(i):=G(\coprod_i D^m)$, $i=0\ldots k$. This sequence has an obvious  $k$-truncated  right action of $\Edm$, respectively $^\delta\Edm$. This is a general fact since the operad in question is the operad of endomorphisms of $D^m$ and $G(\bullet)$ is a sequence of values of a cofunctor on the monoidal powers of $D^m$. Thus $T_kF(M)$ can be described as the space of derived maps of $k$-truncated right modules over $^\delta\Edm$:
$$
T_kF(M)\simeq\hRmod_{^\delta\Edm}{}_{\leq k}\left(^\delta\Emb(\bullet,M),F(\bullet)\right).
\footnote{As a general remark regarding notation, for a right module $G(\bullet)$ we will denote by $G(\bullet\bigr|_{\leq k})$ its $k$-truncation. However if we consider the space of (derived) maps of truncated right modules the notation $\bigr|_{\leq k}$ will be dropped since $\Rmod_{\leq k}$ already indicates that the objects are truncated.}
\eqno(\numb)\label{eq_rmod_descr}
$$
It is natural to ask whether $^\delta\Edm$ can be replaced by $\Edm$ in case $F$ is continuous. I asked this question to M.~Weiss and it turned out that his student P.~Boavida de Brito was already working on the same problem and a few months later they found an elegant solution thus proving the following

\begin{theorem}[\cite{BoavidaWeiss}]\label{t_cont_free}
For a cofunctor $F\colon\Man_m\to\Top$ the natural composition
$$
F(M)\longrightarrow\underset{\Edm}{\Rmod}{}_{\leq k}\left(\Emb(\bullet,M),F(\bullet)\right)
\longrightarrow\underset{\Edm}{\hRmod}{}_{\leq k}\left(\Emb(\bullet,M),F(\bullet)\right)
\eqno(\numb)\label{eq_comp_cont_free}
$$
is equivalent to the homotopy $\J_k$-sheafification $F(M)\to T_kF(M)$.
\end{theorem}

They used a slightly different language to formulate this result, but their~\cite[Section 6]{BoavidaWeiss} shows that it can be reformulated using the operadic approach.
A discrete version of this result appeared in~\cite{ArTur} with the only difference that in that paper we considered only submanifolds of $\R^m$ and the acting operad was the operad of little (non-framed) discs.
The above result is quite remarkable not only because it gives a connection between the theory of operads and manifold calculus, but also it shows that the discretization of the operad of (framed) discs still keeps a lot of information about the initial topological operad. It would be interesting to understand exactly what information is preserved by discretization and to which extend it is true for any or a larger range of topological operads. As a natural analogy the homology  of any Lie group with coefficients in a cyclic group is conjectured to coincide with the homology of its discretization. This conjecture known as the Friedlander-Milnor conjecture was a subject of an extensive research~\cite{FrMi,Milnor,Morel,Sah,Suslin}.

\section{Embedding tower and Fulton-MacPherson operad}\label{s2}
The  motivating example for the Weiss manifold calculus is the study of embedding spaces. The Taylor tower~\eqref{eq_tower} for the embedding cofunctor $\Emb(-,N)$ is usually called {\it embedding tower}. It turns out that one can obtain a nice decription of the embedding tower by replacing the operad $\Edm$ by a much smaller but equivalent operad $\F_m^{fr}$ the so called {\it framed Fulton-MacPherson operad}.  Recall the Fulton-MacPherson operad $\F_m$~\cite{GetzJon,Salvatore}. This operad was simultaneously introduced by  several people, in particular  by Getzler and Jones~\cite{GetzJon}. Its components are manifolds with corners such that the interior of $\F_m(k)$ is $C(k,\R^m)/G$ the configuration space of $m$ distinct labeled points in~$\R^m$ quotiented out by translations and positive rescalings. Notice that $\F_m(0)=\F_m(1)=\{*\}$. By {\it reduced Fulton-MacPherson operad} $\bar{\F}_m$ we will understand the suboperad of $\F_m$ obtained by making the degree zero component to be empty $\bar{\F}_m(0)=\emptyset$ and keeping all the other components the same $\bar{\F}_m(k)=\F_m(k)$, $k\geq 1$.  It is noticed in~\cite{GetzJon} that the operad $\bar{\F}_m$ is cofibrant. As an operad in sets it is freely generated by the interiors of its components. The {\it framed Fulton-MacPherson operad} $\F_m^{fr}$ has components $\F_m(k)\times (\GL_m)^{\times k}$, $k\geq 0$, where $\GL_m$ is the group of general linear transformations of $\R^m$. Each component of this operad is a certain compactification of the space of framed configurations modulo translations and rescaling. The composition in $\F_m^{fr}$ uses the fact that $\F_m$ is an operad in spaces with $\GL_m$-action, see~\cite{Salvatore}. We will also consider the {\it oriented framed Fulton-MacPherson operad} $\F_m^{or}$ which is a suboperad of $\F_m^{fr}$ and whose components are $\F_m(k)\times (\GL_m^+)^{\times k}$, $k\geq 0$, where $\GL_m^+$ is the group of orientation preserving linear transformations of $\R^m$. 

For any manifold $M$ let $C(k,M)$, $k\geq 0$, denote the configuration space
$$
C(k,M)=\left\{ (x_1,\ldots,x_k)\in M^{\times k}\, |\, x_i\neq x_j \text{ for all } i\neq j\right\}.
$$
And let $C[k,M]$ denote its Axelrod-Singer-Fulton-MacPherson compactification~\cite{AxelSing, SinhaCompact}. A thorough treatment of this construction is given by Sinha in~\cite{SinhaCompact} from where we borrowed our notation. $C[k,M]$ is a manifold with corners whose interior is $C(k,N)$. The
 boundary strata consist of configurations where some of the points collided. One has an obvious projection $C[k,M]\to M^{\times k}$ and we denote by $p_i\colon C[k,M]\to M$ its $i$-th component, $1\leq i\leq n$. In case a manifold $N$ has dimension $\geq m$ we define a space $C^{m\_fr}[k,N]$ which fibers over $C[k,N]$ with a fiber over any point $X\in C[k,N]$ being the space of tuples $(\alpha_1,
\ldots,\alpha_k)$, where each $\alpha_i\colon\R^m\hookrightarrow T_{p_i(X)}M$ is a linear injective map called {\it partial framing}. In case of a manifold of dimension $m$ we will simply write $C^{fr}[k,M]$ instead of $C^{m\_fr}[k,M]$. In case $M$ is oriented we also consider spaces $C^{or}[k,M]\subset C^{fr}[k,M]$ for which all the framings $\alpha_i\colon\R^m\stackrel{\simeq}{\longrightarrow}T_{p_i(X)}M$ are orientation preserving.
Obviously, the sequences $C^{fr}[\bullet,M]$, $C^{m\_fr}[\bullet,N]$ are right modules over $\F_m^{fr}$. Similarly, for an oriented $m$-manifold $M$, the sequence $C^{or}[\bullet,M]$ is naturally a right module over $\F_m^{or}$;
and for a parallelized $M$, $C[\bullet,M]$ is a right module over $\F_m$. Notice that any embedding
$f\colon M\hookrightarrow N$ induces a natural {\it evaluation map} $ev_f\colon C^{fr}[k,M]\to C^{m\_fr}[k,N]$ which is a morphism of right $\F_m^{fr}$-modules.

\begin{theorem}\label{t_fult_mac}
In the above notation the composition
$$
\Emb(M,N)\stackrel{ev}{\longrightarrow}\underset{\F_m^{fr}}{\Rmod}{}_{\leq k}\left(C^{fr}[\bullet,M],C^{m\_fr}[\bullet,N]\right)\longrightarrow \underset{\F_m^{fr}}{\hRmod}{}_{\leq k}\left(C^{fr}[\bullet,M],C^{m\_fr}[\bullet,N]\right)
\eqno(\numb)\label{eq_fult_mac}
$$
is equivalent to the $\J_k$-sheafification $\Emb(M,N)\to T_k\Emb(M,N)$. In particular the limit of the embedding tower $T_\infty\Emb(M,N)$ is equivalent to the space of derived maps of right $\F^{fr}_m$-modules
$$
T_\infty\Emb(M,N)\simeq \underset{\F_m^{fr}}{\hRmod}\left(C^{fr}[\bullet,M],C^{m\_fr}[\bullet,N]\right).
$$
In case $M$ is oriented, respectively parallelized, the same is true for the compositions
$$
\Emb(M,N)\longrightarrow \underset{\F_m^{or}}{\hRmod}{}_{\leq k}\left(C^{or}[\bullet,M],C^{m\_fr}[\bullet,N]\right),
\eqno(\numb)\label{eq_fult_mac_or}
$$
$$
\Emb(M,N)\longrightarrow \underset{\F_m}{\hRmod}{}_{\leq k}\left(C[\bullet,M],C^{m\_fr}[\bullet,N]\right).
\eqno(\numb)\label{eq_fult_mac_parall}
$$
\end{theorem}

The first statement of this theorem is equivalent to the Boavida-Weiss Theorem~\ref{t_cont_free} applied to the embedding cofunctor. Our proof of~\ref{t_fult_mac} can be considered as an alternative proof of~\ref{t_cont_free} in this case. To see this equivalence we recall~\cite[Proposition~3.9]{Salvatore} that $\F_m^{fr}$ is equivalent to $\Edm$ via a zigzag
$$
\Edm\longleftarrow W(\Edm)\longrightarrow \F_m^{fr},
$$
where $W(\Edm)$ is the Boardmann-Vogt replacement of $\Edm$. Thus by~\cite[Theorem~16.B]{Fresse} the right-hand sides of~\eqref{eq_comp_cont_free} and~\eqref{eq_fult_mac} can be expressed as spaces of derived maps of (truncated) right modules over $W(\Edm)$. Finally,  one has similar zigzags of equivalences of right $W(\Edm)$-modules
$$
\Emb(\bullet,N)\longleftarrow W(\Emb(\bullet,N))\longrightarrow C^{m\_fr}[\bullet,N],
\eqno(\numb)\label{eq_rmod_zigN}
$$
$$
\Emb(\bullet,M)\longleftarrow W(\Emb(\bullet,M))\longrightarrow C^{fr}[\bullet,M],
\eqno(\numb)\label{eq_rmod_zigM}
$$
where $W(-)$ is a similar Boardmann-Vogt resolution of the corresponding right module over $\Edm$. 

In fact our construction can be used to give an alternative proof of Theorem~\ref{t_cont_free}, see Section~\ref{s5}.

To prove Theorem~\ref{t_fult_mac} we will construct a cofibrant replacement $\widetilde{C}^{fr}[\bullet,M]$ (functorial on $M$) of $C^{fr}[\bullet,M]$ in the category of right modules over $\F_m^{fr}$. The following result is important to understand that construction.

\begin{lemma}\label{l_reduced_cofibr}
For any smooth $m$-manifold $M$, $C^{fr}[\bullet,M]$ is cofibrant in the category of right modules over the reduced framed Fulton-MacPherson operad $\bar{\F}_m^{fr}$. Similarly its any $k$-truncation $C^{fr}[\bullet\bigr|_{\leq k},M]$ is cofibrant in the category of $k$-truncated right modules.
\end{lemma}

\begin{proof}
Intuitively one can see that $C^{fr}[\bullet,M]$ is cofibrant since  as a right $\bar{\F}_m^{fr}$ module in sets it is freely generated by the symmetric  sequence $C(\bullet,M)$ --- the interiors of $C[\bullet,M]$. In the sequel each $C(k,M)$, $k\geq 0$,  will be called a {\it generating stratum} of this right module. Below we give a more rigorous argument.

Let $Sub_{\leq k}(M)$ denote the space of subsets of $M$ of cardinality $\leq k$, topologized as a quotient of $\{\emptyset\}\amalg M^{\times k}/\Sigma_k$. We will also consider the space of all finite subsets of $M$
$$
Sub(M)=\bigcup_{k\geq 0}Sub_{\leq k}(M).
$$
Denote by $g$ the composition
$$
g\colon\coprod_{k=0}^\infty C^{fr}[k,M]\longrightarrow \coprod_{k=0}^\infty M^{\times k} \longrightarrow Sub(M).
$$
For $X\in C^{fr}[k,M]$ we say that $g(X)$ is the {\it set of geometrically distinct points of $X$}. Define a filtration in $C^{fr}[\bullet,M]$ by the number of geometrically distinct points $C^{fr}_{(\ell)}[\bullet,M]=g^{-1}\left(Sub_{\leq \ell}(M)\right)$ :
$$
C^{fr}_{(0)}[\bullet,M]\subset C^{fr}_{(1)}[\bullet,M]\subset C^{fr}_{(2)}[\bullet,M]\subset \ldots
\eqno(\numb)\label{eq_filtration1}
$$
Notice that the right action of $\bar{\F}_m^{fr}$ does not change the image of $g$:
$$
g(X\circ_i c)=g(X)
\eqno(\numb)\label{eq_g_invariance}
$$
for all $X\in C^{fr}[\bullet,M]$ and $c\in \bar{\F}_m^{fr}(\bullet)$. Therefore filtration~\eqref{eq_filtration1} is a filtration of right modules over $\bar{\F}_m^{fr}$. One can show that each inclusion in this filtration is a cofibration, which guarantees that all elements of the filtration and its colimit are cofibrant. In case $M$ is parallelized one has the following pushout square of right $\bar{\F}_m^{fr}$-modules:
$$
\xymatrix{
Free_{\bar{\F}_m^{fr}}\left(\partial\,C[\ell,M];\ell\right)\ar[r]\ar[d]&Free_{\bar{\F}_m^{fr}}\left(C[\ell,M];\ell\right)\ar[d]\\
C^{fr}_{(\ell-1)}[\bullet,M]\ar[r]^(.75){\Bigr\lrcorner}&C^{fr}_{(\ell)}[\bullet,M].
}
\eqno(\numb)\label{eq_pushout1}
$$
In the above $Free_{\bar{\F}_m^{fr}}(A;\ell)$  denotes the free $\bar{\F}_m^{fr}$ right module generated by a space $A$ with a free $\Sigma_{\ell}$ action; $\partial\, C[\ell,M]$ denote the boundary of $C[\ell,M]$. The vertical maps above are defined by a choice of trivialization of $TM$. Since the upper arrow is a cofibration so is the lower one.

In the case $M$ is not parallelized one can consider a cellular decomposition of $M$ and then refine the above argument using the trivialization of the tangent bundle over each cell.
\end{proof}

  The above lemma shows that $C^{fr}[\bullet,M]$  fails to be cofibrant as a right  $\F_m^{fr}$ module only because of  the degree zero component $\F_m^{fr}(0)=*$ which acts by forgetting the corresponding point in  configuration.
A slight adjustment has to be done in order to make it cofibrant.

Define $\widetilde{C}^{fr}[k,M]$ as a {\it space of hairy configurations}. Its points are tuples $(X;y_1,\ldots,y_\ell;t_1,\ldots t_\ell)$, $\ell\geq 0$, where $X\in C^{fr}[k,M]$; $y_i\in M$, $i=1\ldots \ell$; $0\leq t_1\leq t_2\leq\ldots\leq t_\ell\leq 1$. The data $(\bar y; \bar t)=(y_1,\ldots,y_\ell;t_1,\ldots,t_\ell)$ can be viewed  as a bunch of hairs that grow from the points $y_1,\ldots,y_\ell$ and have length $t_1,\ldots,t_\ell$ respectively. Thus
$$
\widetilde{C}^{fr}[k,M]=\left.\left(\coprod_{\ell=0}^\infty C^{fr}[k,M]\times (M^{\times\ell}\times [0,1]^\ell)/\Sigma_\ell\right)\right/\sim,
\eqno(\numb)\label{eq_w_eq}
$$
where the equivalence relation is as follows: If one of the hairs (say $y_1$) gets contracted to zero, the corresponding point $y_1$ disappears. If two hairy points $y_i$ and $y_j$ collide, only the hair of the longer length survives. If one of the hairs collides with a point or a conglumeration of points of $X$, the hair also disappears. Explicitly,
$$
(X;y_1,\ldots,y_\ell;t_1,\ldots,t_\ell)\sim (X;y_{\sigma_1},\ldots,y_{\sigma_\ell};t_{\sigma_1},\ldots,t_{\sigma_\ell}),
\eqno(\numb)\label{eq_hairs_sigma}
$$
whenever $t_{\sigma_1}\leq\ldots\leq t_{\sigma_\ell}$, $\sigma\in\Sigma_\ell$;
$$
(X;y_1,y_2,\ldots,y_\ell;0,t_2,\ldots,t_\ell)\sim (X;y_2,\ldots,y_\ell;t_2,\ldots,t_\ell);
\eqno(\numb)\label{eq_hairs_0}
$$
$$
(X;y_1,\ldots,y_\ell;t_1,\ldots,t_\ell)\sim (X;y_1,\ldots,\hat y_i,\ldots,y_\ell;t_1,\ldots,\hat t_i,\ldots,t_\ell)
\eqno(\numb)\label{eq_hairs_collide}
$$
whenever $y_i=y_j$, $i<j$, or $y_i=p_m(X)$, $1\leq m\leq k$.

The right action of any element $c\in\F_m^{fr}(k)$ in degree $k\geq 1$, is defined to affect only $X$:
$$
(x;\bar y;\bar t)\circ_i c= (x\circ_i c;\bar y;\bar t).
\eqno(\numb)\label{eq_rigt_act1}
$$
For $\{e\}=\F_m^{fr}(0)$ the right action is defined by
$$
(X;\bar y; \bar t)\circ_i e=
(X\circ_i e;\bar y, p_i(X); \bar t, 1).
\eqno(\numb)\label{eq_right_act2}
$$
In other words this action replaces the $i$-th point in $X$ by a hair of length 1.
Notice however that in case $p_i(X)=p_j(X)$ for some $j\neq i$, one has $(X\circ_i e;\bar y, p_i(X); \bar t, 1)=(X\circ_i e;\bar y; \bar t)$ by~\eqref{eq_hairs_collide}.

\begin{proposition}\label{p_cofibr}
The natural projection
$$
\widetilde{C}^{fr}[\bullet,M]\to C^{fr}[\bullet,M],
\eqno(\numb)\label{eq_forget_hairs}
$$
that forgets all hairs, defines a cofibrant replacement of $C^{fr}[\bullet,M]$ as a right $\F_m^{fr}$ module. Moreover for every $k\geq 0$ the $k$-th truncation
$\widetilde{C}^{fr}_{(k)}[\bullet\bigr|_{\leq k},M]$ of the $k$-th filtration term~\eqref{eq_filtration2} is a cofibrant replacement   of  $C^{fr}[\bullet\bigr|_{\leq k},M]$ as a $k$-truncated right $\F_m^{fr}$ module.
\end{proposition}

By an obvious analogy we define $\widetilde{C}^{or}[\bullet,M]$, respectively $\widetilde{C}[\bullet,M]$, as a cofibrant replacement of $C^{or}[\bullet,M]$, respectively $C[\bullet,M]$, in the category of right modules over $\F_m^{or}$, respectively $\F_m$.

\begin{proof}[Sketch of the proof]
First we notice that the projection~\eqref{eq_forget_hairs} is an equivalence of $\F_m^{fr}$-modules. This means that for every component $\bullet=\ell$ this map is a $\Sigma_\ell$-equivariant homotopy equivalence. The homotopy inverse  is the natural inclusion $C^{fr}[\ell,M]\hookrightarrow \widetilde{C}^{fr}[\ell,M]$.

Secondary, it is easy to see that $\widetilde{C}^{fr}[\bullet,M]$ as a right $\F_m$-module in sets is freely generated by the symmetric sequence
$$
\widetilde{C}(\ell,M)=\coprod_{i=0}^{+\infty} \left. C(\ell+i,M)\times (0,1)^i \right/ \Sigma_i, \quad \ell\geq 0,
\eqno(\numb)\label{eq_gen_strat}
$$
where $\Sigma_i$ acts by a simultaneous permutation of the last $i$ points in $C(\ell +i,M)$ and the coordinates of $(0,1)^i$. We sketch a proof below that shows in which order the above generating strata are attached. For simplicity one can assume that $M$ is parallelized. If it is not the argument must be further refined as in the proof of Lemma~\ref{l_reduced_cofibr}.  Define a map
$$
\widetilde{g}\colon\coprod_{\ell=0}^{+\infty} \widetilde{C}^{fr}[\ell,M]\to Sub(M)
\eqno(\numb)\label{eq_tilde_g}
$$
by sending $\widetilde{X}=(X;y_1\ldots y_i;t_1\ldots t_i)$, with all $t_j\neq 0$, to $g(X)\bigcup\{y_1\ldots y_i\}$. Notice that $\widetilde{g}$ is not continuous contrary to the map $g$. Again $\widetilde{g}(\widetilde{X})$ is caled the {\it  set of geometrically distinct points} of $\widetilde{X}$. Notice that this map is invariant with respect to the $\F_m^{fr}$ action:
$$
\widetilde{g}(\widetilde{X}\circ_i c)=\widetilde{g}(\widetilde{X})
$$
for every $X\in \widetilde{C}^{fr}[\bullet,M]$ and $c\in \F_m^{fr}$. In particular this means we can define a filtration of right $\F_m^{fr}$-modules:
$$
\widetilde{C}_{(0)}^{fr}[\bullet,M]\subset \widetilde{C}_{(1)}^{fr}[\bullet,M]\subset \widetilde{C}_{(2)}^{fr}[\bullet,M]\subset\ldots,
\eqno(\numb)\label{eq_filtration2}
$$
where $\widetilde{C}_{(k)}^{fr}[\bullet,M]$ is the preimage $\widetilde{g}^{-1}\left(Sub_{\leq k}(M)\right)$. Each inclusion in the above filtration is a cofibration  of right $\F_m^{fr}$ modules. To see
 this we notice that besides the filtration by the number of geometrically distinct points there is another natural filtration in $\widetilde{C}^{fr}[\bullet,M]$ by the number of hairs whose length is strictly between
0 and 1. Let $\widetilde{C}^{fr}_{(k),i}[\bullet,M]$, $k\geq 0$, $i\geq -1$, denote the right submodule of $\widetilde{C}_{(k)}^{fr}[\bullet,M]$ that consists of hairy configurations with either $\leq k$
geometrically distinct points or with exactly $k+1$ geometrically distinct points, but with $\leq i$ hairs of length strictly between 0 and 1. One can easily see that the generating stratum $\left. C(\ell+i,M)\times
(0,1)^i\right/\Sigma_\ell$ from~\eqref{eq_gen_strat} is attached exactly when we we pass from $\widetilde{C}^{fr}_{(\ell+i-1),i-1}[\bullet,M]$ to $\widetilde{C}^{fr}_{(\ell+i-1),i}[\bullet,M]$. By this we
mean that one has a pushout square similar to~\eqref{eq_pushout1}:
$$
\xymatrix{
Free_{{\F}_m^{fr}}\left(\partial\,\left(C[\ell+i,M]\times [0,1]^i/\Sigma_i\right);\ell\right)\ar[r]\ar[d]&Free_{{\F}_m^{fr}}\left(C[\ell+i,M]\times [0,1]^i/\Sigma_i;\ell\right)\ar[d]\\
\widetilde{C}^{fr}_{(\ell+i-1),i-1}[\bullet,M]\ar[r]^(.75){\Bigr\lrcorner}&\widetilde{C}^{fr}_{(\ell+i-1),i}[\bullet,M].
}
\eqno(\numb)\label{eq_pushout2}
$$

The truncated case follows from the fact that the generating strata of $\widetilde{C}_{(k)}^{fr}[\bullet,M]$ all lie in the components of degree $\leq k$.
\end{proof}

Theorem~\ref{t_fult_mac} is a consequence of Propositions~\ref{p_cofibr}, \ref{p_hom_sheaf}, \ref{p_eval} and Lemma~\ref{l_homeo}.

\begin{proposition}\label{p_hom_sheaf}For any right $\F_m$ module $\calN(\bullet)$,
the cofunctor that assigns to any $U\in \calO(M)$ (respectively $U\in {}^\delta\Man_m$) the space
$$
\underset{\F_m^{fr}}{\Rmod}{}_{\leq k}\left(\widetilde{C}^{fr}_{(k)}[\bullet,U],\calN(\bullet)\right)
\eqno(\numb)\label{eq_rmod_maps}
$$
is a homotopy $\J_k$-sheaf on $\calO(M)$ (respectively ${}^\delta\Man_m$).
\end{proposition}

We prove this result in Section~\ref{s3}.  The space~\eqref{eq_rmod_maps} looks almost like the space of sections of a  stratified fiber bundle over the filtered space
$$
Sub_{\leq 0}(U)\subset Sub_{\leq 1} (U) \subset Sub_{\leq 2} (U) \subset Sub_{\leq 3} (U)\subset\ldots
\subset Sub_{\leq k}(U).
$$
Our argument is thus a slight adjustment of the proof of a similar result that the functor $U\mapsto Maps(U^{\times k},X)$ is polynomial of degree $\leq k$, see \cite[Proposition~3.1]{GKW}.

\begin{proposition}\label{p_eval}
The composition
$$
\Emb(U,N)\stackrel{ev}{\longrightarrow}\underset{\F_m^{fr}}{\Rmod}{}_{\leq k}\left(C^{fr}[\bullet,U],C^{m\_fr}[\bullet,N]\right)\longrightarrow \underset{\F_m^{fr}}{\Rmod}{}_{\leq k}\left(\widetilde{C}_{(k)}^{fr}[\bullet,U],C^{m\_fr}[\bullet,N]\right)
\eqno(\numb)\label{eq_eval}
$$
is a homotopy equivalence whenever $U$ is a disjoint union of $\leq k$ $m$-balls.
\end{proposition}

The proof is given in Section~\ref{s4}.

\begin{lemma}\label{l_homeo}
In case $M$ is orientable, respectively parallelized, one has a natural homeomorphism of spaces
$$
\underset{\F_m^{fr}}{\Rmod}{}_{\leq k}\left(\widetilde{C}_{(k)}^{fr}[\bullet,M],C^{m\_fr}[\bullet,N]\right)\cong
\underset{\F_m^{or}}{\Rmod}{}_{\leq k}\left(\widetilde{C}_{(k)}^{or}[\bullet,M],C^{m\_fr}[\bullet,N]\right),
$$
respectively
$$
\underset{\F_m^{fr}}{\Rmod}{}_{\leq k}\left(\widetilde{C}_{(k)}^{fr}[\bullet,M],C^{m\_fr}[\bullet,N]\right)\cong
\underset{\F_m}{\Rmod}{}_{\leq k}\left(\widetilde{C}_{(k)}[\bullet,M],C^{m\_fr}[\bullet,N]\right).
$$
\end{lemma}

This lemma is obvious by inspection.

\section{Proof of Proposition~\ref{p_hom_sheaf}}\label{s3}
We need to show that the functor that assigns to any open set $U$ the space
$$
R_k(U,\calN):=\underset{\F_m^{fr}}{\Rmod}{}_{\leq k}\left(\widetilde{C}_{(k)}^{fr}[\bullet,U],\calN(\bullet)\right)
$$
is a homotopy sheaf on $\calO(M)$ with respect to the Grothendieck topology $\J_k$. This means that for any cover $\{U_i\subset U\}_{i\in I}$ such that $\bigcup_{i\in I} U_i^{\times k}=U^{\times k}$ one has that the natural map
$$
R_k(U,\calN)\stackrel{\simeq}{\longrightarrow}\holim_{\emptyset\neq S\subset I} R_k(U_S,\calN)
$$
is a weak equivalence. In the above the homotopy limit is taken over the category of finite non-empty subsets of $I$, and $U_S=\bigcap_{i\in S} U_i$.

Let us prove first that the functor
$$
U\mapsto\bar{R}_k(U,\calN):=\underset{\bar{\F}_m^{fr}}{\Rmod}{}_{\leq k}\left(C^{fr}[\bullet,U],\calN(\bullet)\right)
$$
is polynomial of degree $\leq k$. We don't need this result, but technically it is easier, and the proof of the statement that we need is just a slight modification of the argument given below.

For a set $J$ denote by $\Delta^J$ its {\it formal convex hull}. It consists of linear combinations $\vec{\lambda}=\sum_{i\in J} \lambda_i\langle i\rangle$ of elements in $J$, such that $\sum_{i\in J}\lambda_i=1$; $0\leq \lambda_i\leq 1$, $i\in J$; and the support of $\vec{\lambda}$ is finite:
$$
supp(\vec{\lambda})=\{i\in J\, |\, \lambda_i\neq 0\}<\infty.
$$
In case $J$ is finite $\Delta^J$ is a simplex. In general case it is the realization of the full combinatorial simplicial complex on the vertex set $J$. In particular we have that $\Delta^J$ is naturally a $CW$-complex. The space  $\holim_{\emptyset\neq S\subset I} \bar{R}_k(U_S,\calN)$ can be described as the space of natural transformations between the functor that assigns $\Delta^S$ to any finite non-empty set $S\subset I$ and the functor that assigns $\bar{R}_k(U_S,\calN)$ to $S\subset I$. Thus a point $G$ in the homotopy limit is given by a family of maps
$$
G_S\colon\Delta^S\to \bar R_k(U_S,\calN),\quad \emptyset\neq S\subset I.
$$
By adjunction this family of maps can be written as another collection of maps
$$
G_{S,k}\colon \Delta^S\times C^{fr}[k,U_S]\to \calN(k), \quad \emptyset\neq S\subset I,\quad k\geq 0,
$$
that satisfy certain boundary conditions. In particular for $S_1\subset S_2$, one has $\Delta^{S_1}\subset\Delta^{S_2}$, $U_{S_1}\supset U_{S_2}$, and
$$
G_{S_2,k}\Bigr|_{\Delta^{S_1}\times C^{fr}[k,U_{S_2}]}=G_{S_1,k}\Bigr|_{\Delta^{S_1}\times C^{fr}[k,U_{S_2}]}.
$$
For this reason we drop the subindices $S$ and $k$ and will simply write $G(\vec{\lambda},X)$, where $\vec{\lambda}\in\Delta^I$, and $X\in C^{fr}[k,U]$ for some $k\geq 0$. Notice that $G(\vec{\lambda},X)$ is defined if and only if $g(X)\subset U_{supp(\vec{\lambda})}$.

One has a natural inclusion
$$
i\colon \bar{R}_k(U,\calN){\longrightarrow}\holim_{\emptyset\neq S\subset I} \bar{R}_k(U_S,\calN)
$$
that sends $F$ to $(iF)(\vec{\lambda},X)=F(X)$. Let us describe the homotopy inverse to $i$. Recall $Sub_{\leq k}(U)$. It is homeomorphic to a $CW$-complex and therefore is paracompact. One also has that $Sub_{\leq k}(U_i)$, $i\in I$, is an open cover of $Sub_{\leq k}(U)$, since $\bigcup_{i\in I} U_i^{\times k}=U^{\times k}$. Let $\vec{\psi}=\sum_{i\in I}\psi_i\langle i\rangle$ be a partition of unity on $Sub_{\leq k}(U)$ subordinate to the above cover. We view it as a continuous map
$$
\vec{\psi}\colon  Sub_{\leq k} (U)\to \Delta^I
$$
that has the property $g\subset U_{supp(\vec{\psi}(g))}$ for any $g\in Sub_{\leq k}(U)$. On the other hand, we also have that the map
$$
g\colon\coprod_{i=0}^k C^{fr}[i,U]\to Sub_{\leq k}(U)
$$
is continuous. Now for $G\in \holim_{\emptyset\neq S\subset I} \bar{R}_k(U_S,\calN)$ viewed as a function $G(\vec{\lambda},X)$, we define $s(G)\in \bar{R}_k(U,\calN)$ by the formula
$$
s(G)(X)=G(\vec{\psi}(g(X)),X).
\eqno(\numb)\label{eq_inv_homot1}
$$
Since the right action of $\bar{\F}_m^{fr}$ preserves $g(X)$, see~\eqref{eq_g_invariance}, we get that $s(G)$ is a morphism of right $\bar{\F}_m^{fr}$-modules. It is easy to see that $s\circ i$ is identity, whereas $s\circ i$ sends $G(\vec{\lambda},X)$ to
$$
(si)(G)(\vec{\lambda},X)=G(\vec{\psi}(g(X)),X).
$$
The homotopy between $G$ and $(si)(G)$ is given by
$$
G(\tau\cdot \vec{\psi}(g(X))+ (1-\tau)\cdot\vec{\lambda},X),\quad 0\leq \tau\leq 1.
$$
Thus we proved that $\bar{R}_k(-,\calN)$ is a homotopy $\J_k$-sheaf.

Now, let us establish the same result for $R_k(-,\calN)$. One still has an obvious map
$$
i\colon R_k(U,\calN) \longrightarrow\holim_{\emptyset\neq S\subset I} R_k(U_S,\calN)
\eqno(\numb)\label{eq_can_map}
$$
defined as
$$
(iF)(\vec{\lambda},\widetilde{X})=F(\widetilde{X}),
$$
where $\widetilde{X}\in\widetilde{C}^{fr}_{(k)}[j,U]$, $0\leq j\leq k$.
For the homotopy inverse, unfortunately the formula~\eqref{eq_inv_homot1} does not work since the analogous map
$$
\widetilde{g}\colon \coprod_{i=0}^k \widetilde{C}_{(k)}^{fr}[i,U]\to Sub_{\leq k}(U),
$$
that assigns the set of geometrically distinct points, is not continuous anymore. To remedy this we first introduce the map
$$
\Xi\colon\widetilde{C}^{fr}[\bullet,U]\to \widetilde{C}^{fr}[\bullet,U]
$$
that removes all hairs of $\widetilde{X}\in \widetilde{C}^{fr}[k,U]$ of  length $t\leq \frac 12$, and contracts every hair to the length $2t-1$ if its length $t\geq \frac 12$. It is easy to see that $\Xi$ is an endomorphism of $\widetilde{C}^{fr}[k,U]$ as a right $\F_m^{fr}$-module. Moreover $\Xi$ preserves filtration~\eqref{eq_filtration2} and is homotopic to the identity in the space of filtration  preserving endomorphisms.

Now let $\widetilde{X}\in\widetilde{C}^{fr}[k,U]$ has the form $(X;y_1\ldots y_\ell;t_1\ldots t_\ell)$ where
$0<t_1\leq t_2\leq \ldots\leq t_h\leq \frac 12<  t_{h+1}
\leq\ldots\leq t_\ell\leq 1$. Define $\widetilde{g}_i(\widetilde{X})=g(X)\bigcup\{y_{i+1},y_{i+2},\ldots,y_\ell\}$. One obviously has
$$
\widetilde{g}(\widetilde{X})=\widetilde{g}_0(\widetilde{X})\supset \widetilde{g}_1(\widetilde{X})
\supset\ldots\supset \widetilde{g}_h(\widetilde{X})=\widetilde{g}(\Xi(\widetilde{X})).
$$
Notice that $\widetilde{g}_i$ are not uniquely defined, and therefore are also discontinuous, in case $\widetilde{X}$ has several hairs of the same length. For example if $t_i=t_{i+1}$ and all other $t_j$ are different there is a choice which hair we take as the $i$th one and which we take as the $(i+1)$st. Thus in this particular case the set $\widetilde{g}_i(\widetilde{X})$ is not uniquely defined, but however all the other sets including $\widetilde{g}_{i-1}(\widetilde{X})$ are defined uniquely. Now define a map
$$
\vec{\phi}\colon\coprod_{i=0}^k\widetilde C_{(k)}^{fr}[i,U]\to\Delta^I
$$
as follows
$$
\vec{\phi}(\widetilde{X})=\sum_{i=0}^h2(t_{i+1}-t_i)\vec{\psi}(\widetilde{g}_i(\widetilde{X})),
$$
where $t_0=0$ and  abusing notation  $t_{h+1}=\frac 12$. Thus the sum of coefficients $\sum_{i=0}^h2(t_{i+1}-t_i)=1$. We argue below that $\vec{\phi}$ is continuous. Recalling~\eqref{eq_w_eq}
$$
\widetilde{C}_{(k)}^{fr}[i,U]=\left.\left(\coprod_{j=0}^{k-i}C^{fr}[i,U]\times\left(U^{\times j}\times [0,1]^j\right)/\Sigma_j\right)\right/\sim.
$$
Thus we are left to check that the equivalence relations~\eqref{eq_hairs_sigma}, \eqref{eq_hairs_0}, \eqref{eq_hairs_collide} are respected by $\vec{\phi}$, which is an easy exercise.

Now we  are ready to define a map $s$  homotopy inverse to~\eqref{eq_can_map}:
$$
(sG)(\widetilde{X})=G\left(\vec{\phi}(\widetilde{X}),\Xi(\widetilde{X})\right).
$$
This formula is well defined since $\widetilde{g}(\Xi(\widetilde{X}))\subset U_{supp(\vec{\phi}(\widetilde{X}))}$. Indeed, $\widetilde{g}(\Xi(\widetilde{X}))=\widetilde{g}_h(\widetilde{X})\subset
\widetilde{g}_i(\widetilde{X})\subset U_{supp(\vec{\psi}(\widetilde{g}_i(\widetilde{X})))}$,
for $0\leq i\leq h$. On the other hand $\bigcap_{i=0}^h U_{supp(\vec{\psi}(\widetilde{g}_i(\widetilde{X})))} \subset  U_{supp(\vec{\phi}(\widetilde{X}))}$.

One also has that $sG$ is a morphism of right modules over $\F_m^{fr}$ since the right $\F_m^{fr}$ action
does not change $\vec{\phi}(\widetilde{X})$:
$$
\vec{\phi}(\widetilde{X}\circ_i c)=\vec{\phi}(\widetilde{X}),
$$
for every $c\in \F_m^{fr}(\bullet)$, since again each $\widetilde{g}_i$ is preserved by this action.

We check that $s$ is a homotopy inverse to $i$. One has $(si)(F)=F\circ \Xi$. Since $\Xi$ is homotopic to the identity $(si)$ is so. For the opposite composition
$$
(is)(G)(\vec{\lambda},\widetilde{X})=G\left(\vec{\phi}(\widetilde{X}),\Xi(\widetilde{X})\right).
$$
The homotopy
$$
G(\tau\cdot \vec{\lambda}+(1-\tau)\cdot\vec{\phi}(\widetilde{X}),\Xi(\widetilde{X})), \quad 0\leq \tau\leq 1,
$$
shows that $(is)$ is homotopic to the map that sends $G(\vec{\lambda},\widetilde{X})$ to $H(\vec{\lambda},\widetilde{X})=G(\vec{\lambda},\Xi(\widetilde{X}))$. Finally using the homotopy between $\Xi$ and the identity we see that $(is)$ is also homotopic to the identity.

This finishes the proof of Proposition~\ref{p_hom_sheaf}.

\section{Proof of Proposition~\ref{p_eval}}\label{s4}
We need to show that the natural evaluation map $ev_k\colon\Emb(U,N)\to \underset{\F_m^{fr}}\Rmod{}_{\leq k}\left(\widetilde{C}^{fr}_{(k)}[\bullet,U], C^{m\_fr}[\bullet,N]\right)$ is a homotopy equivalence whenever $U$ is a disjoint union of $\ell$ balls with $\ell\leq k$. Let $L\subset U$ be a finite subset of $U$ with exactly one point in each connected component. We fix a bijection $b\colon\{1\ldots\ell\}\to L$ and also framings $\alpha_i\colon\R^m\stackrel{\simeq}{\longrightarrow}
T_{b(i)}U$ for each point in $L$. We denote by $L^{fr}$ the corresponding point in $C^{fr}(\ell,M)$. One has a natural evaluation map
$$
Ev_{L^{fr}}\colon\Emb(U,N)\stackrel{\simeq}{\longrightarrow} C^{m\_fr}(L,N),
$$
that sends $f\in\Emb(M,N)$ to the configuration $f(b(1))\ldots f(b(\ell))$ with framings defined as compositions $\R^m\stackrel{\alpha_i}{\longrightarrow}T_{b(i)}U\stackrel{f_*}{\longrightarrow} T_{f(b(i))}N$. One can easily see that this map is a homotopy equivalence, see for example~\cite{GKW}.

Let $C^{fr}[\bullet,L]$ denote the right $\F_m^{fr}$ submodule of $\widetilde{C}^{fr}[\bullet,U]$ generated by $L^{fr}\in C^{fr}[\ell,L]$. It is easy to see that this submodule is naturally homeomorphic to a free right module generated by $\Sigma_\ell$  in degree $\ell$. In other words the map
of right $\F_m^{fr}$ modules
$$
Free_{\F_m^{fr}}(\Sigma_\ell,\ell)\to C^{fr}[\bullet,L]
$$
that sends the unit of $\Sigma_\ell$ to $L^{fr}$, is a homeomorphism.  Therefore the $k$-th truncation
$C^{fr}[\bullet\bigr|_{\leq k},L]$ is also a truncated right module freely generated by $\Sigma_\ell$ in degree $\ell$ (here we use the fact that $\ell\leq k$).

The inclusion
$$
C^{fr}[\bullet,L]\subset \widetilde{C}^{fr}[\bullet,U]
$$
is a homotopy equivalence of right $\F_m^{fr}$ modules. The homotopy inverse is obtained by contracting each disc in $U$ to the corresponding point in $L$, thus  sending usual configurations to infinitesimal configurations. Notice that the same is true for the inclusion of the truncations
$$
C^{fr}[\bullet\bigr|_{\leq k},L]\subset \widetilde{C}^{fr}_{(k)}[\bullet\bigr|_{\leq k},U].
$$

As a result we get a sequence of homotopy equivalences
\begin{multline}
\underset{\F_m^{fr}}\Rmod{}_{\leq k}\left(\widetilde{C}^{fr}_{(k)}[\bullet,U], C^{m\_fr}[\bullet,N]\right)
\stackrel{\simeq}{\longrightarrow} \underset{\F_m^{fr}}\Rmod{}_{\leq k}\left(C^{fr}[\bullet,L], C^{m\_fr}[\bullet,N]\right) \stackrel{\cong}{\longrightarrow}\\
\stackrel{\cong}{\longrightarrow}
\Maps_{\Sigma_\ell}(\Sigma_{\ell},C^{m\_fr}[\ell,N])
\stackrel{\cong}{\longrightarrow}
C^{m\_fr}[\ell,N].
\label{eq_long_composition}
\end{multline}
The last two maps in~\eqref{eq_long_composition} are homeomorphisms.
Finally we notice that the diagram
$$
\xymatrix{
\Emb(U,N)\ar[r]\ar[d]^\simeq_{Ev_{L^{fr}}}& \underset{\F_m^{fr}}\Rmod{}_{\leq k}\left(\widetilde{C}^{fr}_{(k)}[\bullet,U], C^{m\_fr}[\bullet,N]\right)\ar[d]^\simeq\\
C^{m\_fr}(\ell,N)\ar@{^{(}->}[r]^{\simeq}& C^{m\_fr}[\ell,N]
}
$$
is commutative (the right arrow is the composition~\eqref{eq_long_composition}). We conclude that the top arrow must be an equivalence since all the other maps are.

\section{Proof of Theorem~\ref{t_cont_free}}\label{s5}
By the universal property of polynomial functors~\cite{Weiss} one only needs to show that the functor
$$
U\mapsto \underset{\Edm}\hRmod{}_{\leq k}\left(\Emb(\bullet,U), F(\bullet)\right),
\quad U\in{}^\delta\Man_m,
\eqno(\numb)\label{eq_rmod_functor}
$$
is polynomial of degree $\leq k$, and also that for every $U$ which is a disjoint union of $\leq k$ balls this functor produces a space naturally equivalent to $F(U)$. Notice that the latter statement is straighforward from the Yoneda lemma and also the fact that the category of $k$-truncated right modules over $\Edm$ is equivalent to the category of contravariant functors from  the category $\calO_{\leq k}$.
Due to the zigzag of equivalences of operads~\cite{Salvatore}:
$$
\Edm\longleftarrow W(\Edm)\longrightarrow \F_m^{fr},
\eqno(\numb)\label{eq_zigzag_operads}
$$
and the zigzag of right modules over  $W(\Edm)$
$$
\Emb(\bullet,U)\longleftarrow W(\Emb(\bullet,U))\longrightarrow C^{fr}[\bullet,U],
$$
which is natural in  $U$, the functor~\eqref{eq_rmod_functor}
is equivalent to a similar functor
$$
U\mapsto \underset{\F_m^{fr}}\hRmod{}_{\leq k}\left(C^{fr}[\bullet,U], \ind(F)(\bullet)\right),
\eqno(\numb)\label{eq_ind_functor}
$$
where $\ind(F)$ is a certain right $\F_m^{fr}$ module obtained from $F(\bullet)$ by a natural restriction-extension construction along the zigzag~\eqref{eq_zigzag_operads}, see~\cite[Theorem~16.B]{Fresse}. By Proposition~\ref{p_hom_sheaf} the functor~\eqref{eq_ind_functor} is polynomial of degree~$\leq k$, therefore the equivalent functor~\eqref{eq_rmod_functor} is so.

\section{Spaces of long embeddings}\label{s6}
Theorem~\ref{t_cont_free} has a particularly attractive form for the spaces of higher
dimensional long knots.
Let $\Ebarmn$ denote the homotopy fiber of the inclusion
$$
\Emb_c(\R^m,\R^n)\hookrightarrow \Imm_c(\R^m,\R^n),
\eqno(\numb)\label{eq_emb_imm}
$$
where $\Emb_c(\R^m,\R^n)$, respectively $\Imm_c(\R^m,\R^n)$, is the space of
embeddings $\R^m\hookrightarrow\R^n$, respectively immersions
$\R^m\looparrowright\R^n$, coinciding with a fixed linear embedding $i\colon\R^m\hookrightarrow\R^n$ outside a
compact subset of $\R^m$.
We view $\Ebar_c(-,\R^N)$ as a cofunctor $\widetilde\calO(\R^m)\to\Top$ from the category of open sets of $\R^m$ whose complement is compact. Define $\widetilde\calO_{\leq k}(\R^m)$ as its full subcategory that consists of disjoint unions of $\leq k$ balls and one complement to a closed ball. For an isotopy invariant cofunctor $F\colon\widetilde\calO(\R^m)\to\Top$, its $k$-th Taylor approximation $T_kF$ is the homotopy right Kan extension of $F$ from $\widetilde\calO_{\leq k}(\R^m)$ to $\widetilde\calO(\R^m)$.

To take into account the behavior of embeddings at infinity, we will express spaces
of such embeddings as spaces of derived morphisms of certain {\it infinitesimal bimodules}
over the little discs operad.\footnote{In some previous works the author was using the term {\it weak bimodules} for this notion as for example in~\cite{ArTur}.}  An infinitesimal bimodule over an operad is defined in the
following way. Let $\{O(i)\}$ be an operad. An infinitesimal bimodule over $O$ is a symmetric
sequence $\{M(i), \, i\geq 0\}$ equipped with structure maps (where $i, j\ge 0$,  $1\le s\le i$,
and $\otimes$ stands for a symmetric monoidal product):
$$\circ_s\colon O(i)\otimes M(j)\longrightarrow  M(i+j-1)\quad \text{left action}$$
and
$$\circ'_s\colon M(i)\otimes O(j)\longrightarrow M(i+j-1)\quad\text{right action}$$
satisfying certain rather easily guessed associativity axioms~\cite{ArTur}. For example, left and right actions must be compatible:
$$
(o_1\circ_p m)\circ_q o_2=(o_1\circ_q o_2)\circ_{p+q-1} m,\quad 1\leq q<p\leq i,
$$
$$
(o_1\circ_pm)\circ_{q+k-1}o_2=(o_1\circ_q o_2)\circ_p m, \quad 1\leq p<q\leq i,
$$
for all $o_1\in O(i)$, $o_2\in O(j)$, and $m\in M(k)$. As above the result of composition $\circ_i(o,m)$, and $\circ'_i(m,o)$, for $o\in O(n)$, and $m\in M(k)$, is denoted by $o\circ_i m$, and $m\circ_i o$. Graphically one can view elements of $O$ and $M$ as  having a bunch of inputs and one output. In this representation the composition is shown by the usual grafting picture:


\begin{center}
\psfrag{OM}[0][0][1][0]{$o\circ_3 m$}
\psfrag{MO}[0][0][1][0]{$m\circ_2 o$}
\psfrag{o}[0][0][1][0]{$o$}
\psfrag{m}[0][0][1][0]{$m$}
\includegraphics[width=13cm]{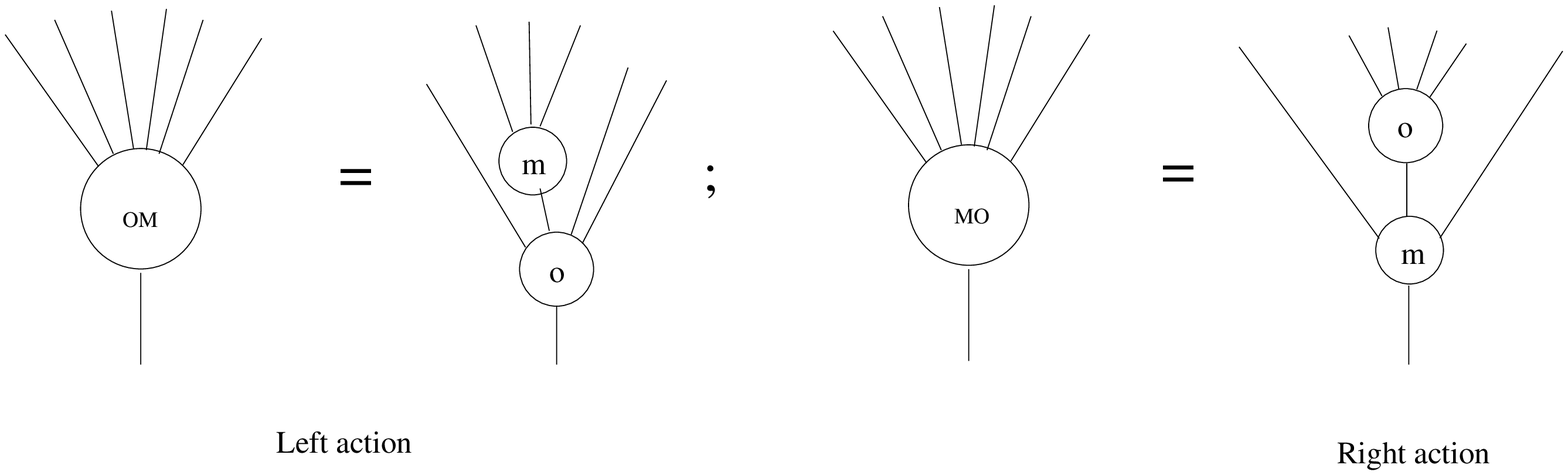}
\end{center}

 By a $k$-truncated infinitesimal bimodule over $O$ we will understand a symmetric sequence $\{M(i),\, i=0\ldots k\}$, with the above structure maps in the range where they can be defined.

As example an infinitesimal bimodule over the commutative operad is the same thing as a contravariant functor from the category $\Gamma$ of finite pointed sets. An infinitesimal bimodule over the non-$\Sigma$ associative operad is nothing but a cosimplicial object.

 As another  example relevant to us, the linear inclusion
$i\colon\R^m\hookrightarrow\R^n$ induces an inclusion of operads of little discs   $\calB_m\hookrightarrow\calB_n$, and thus both $\calB_m$ and $\calB_n$ are infinitesimal bimodules over $\calB_m$.

The category of (truncated) infinitesimal bimodules has all the pleasant formal properties of right
modules. For example, the category of infinitesimal bimodules with values in chain complexes
is (in contrast with the category of honest left modules) an abelian category with
enough projectives. Another nice property of this structure is that an equivalence of operads induces restriction and extension functors which are Quillen equivalences, similarly to the case of right modules~\cite[Theorem~16.B]{Fresse}. Let
$\hIbimod(-,-)$,  $\hIbimod_{\leq k}(-,-)$ denote the space
of derived morphisms between infinitesimal bimodules and  $k$-truncated infinitesimal bimodules respectively. 

\begin{theorem}\label{t_model_for_long_embeddings} One has natural equivalences
$$
T_k\Ebarmn\simeq \underset{\calB_m}{\operatorname{hIbimod}}{}_{\le k}(\calB_m,
\calB_n),\,\, n\geq m.
$$
In particular in the case $n> m+2$
$$\Ebar_c(\R^m,\R^n)\simeq \underset{\calB_m}{\operatorname{hIbimod}}(\calB_m,
\calB_n).$$
\end{theorem}

A discrete version of this theorem appeared in~\cite{ArTur}.
In the last statement of the theorem we use the unpublished result of Goodwillie, Klein, and Weiss about the convergence of the embedding tower in codimension~$\geq 3$.
This statement  generalizes Sinha's
production~\cite{Sinha-OKS} of a cosimplicial space $\calK_n^\bullet$ whose homotopy
totalization $\Tot \calK_n^\bullet$ is weakly equivalent to $\Ebar_c(\R^1,\R^n)$, $n\geq 4$.
This cosimplicial object arises from an operad $\calK_n$ equipped with a map $\Assoc\to
\calK_n$, where $\Assoc$ is the associative operad (which is equivalent to
$\calB_1$), $\calK_n$ is an operad equivalent to $\calB_n$, and the map in question
is equivalent to the usual inclusion $\calB_1\to\calB_n$. As it was already mentioned a
cosimplicial space amounts exactly to an infinitesimal bimodule over $\Assoc$ in the category
of spaces, and that for $m=1$ our theorem above is the same as Sinha's formula:
$$
\Ebar_c(\R^1,\R^n)\simeq\underset{\calB_1}{\operatorname{hIbimod}}(\calB_1,
\calB_n)\simeq
\underset{\Assoc}{\operatorname{hIbimod}}(\Assoc, \calK_n)\simeq \Tot \calK_n^\bullet.
$$

Theorem~\ref{t_model_for_long_embeddings} follows from Theoem~\ref{t_model_long2} and Lemma~\ref{l_inf_bimod_equiv}, and also from the fact that the inclusion of operads $\calB_m\hookrightarrow\calB_n$ is equivalent to the inclusion $\F_m\hookrightarrow\F_n$, which means that there is a zigzag of morphisms of operads in which every horizontal arrow is an equivalence:
$$
\xymatrix{
\calB_m\ar@{^{(}->}[d]&{*}\ar[l]_\simeq\ar[d]\ar[r]^\simeq&\F_m\ar@{^{(}->}[d]\\
\calB_n& {*}\ar[l]_\simeq\ar[r]^\simeq&\F_n.
}
$$
In fact the middle map can be chosen to be $W(\calB_m)\hookrightarrow W(\calB_n)$, see~\cite{Salvatore}.

Consider the sequence
$$
C_*[\bullet,S^n]=\{C_*[k,S^n],\, k\geq 0\},
$$
where $C_*[k,S^n]$ is the Fulton-MacPherson-Axelrod-Singer compactification of the configuration space of $k+1$ distinct points in $S^n$ labeled by $\{*,1,2,\ldots,k\}$, one of which, labeled by~$*$, is fixed to be $\infty\in S^n$. Here and below we view $S^n$ as a one-point compactification of $\R^n$. Thus the interior of $C_*[k,S^n]$ is naturally identified with the configuration space $C(k,\R^n)$. It turns out that $C_*[\bullet,S^n]$ is naturally an infinitesimal bimodule over $\F_n$ (and therefore over $\F_m$, $m\leq n$). Let $p_i\colon C_*[k,S^n]\to S^n$, $i=1\ldots k$, be natural projections.  To define a right $\F_n$ action on this sequence we fix a framing of each projection $p_i(X)\in S^n$, $i=1\ldots k$, $X\in C_*[k,S^n]$. In case $p_i(X)\in\R^n$, the framing $\alpha_i\colon\R^n\to T_{p_i(X)}S^n$ is fixed to be the natural identification $\R^n\cong T_{p_i(X)}\R^n\cong T_{p_i(X)}S^n$.   For the points $X$ on the boundary of $C_*[k,S^n]$ the framing is extended by continuity. It is quite easy to see that in case $p_i(X)=\infty$ this framing of $p_i(X)$ depends only on the direction from which $p_i(X)$ approaches $*$. Those framings enable $C_*[\bullet,S^n]$ with a right $\F_n$ action. Indeed, the right action of $\F_n$ replaces the corresponding point in the configuration $X$ by an infinitesimal configuration $c\in\F_n(\bullet)$ inserted accordingly to the framing. The infinitesimal left action produces so called \lq\lq strata at infinity". Let $s$ be the inversion map considered as a coordinate chart at $\infty$:
$$
s\colon\{\infty\}\bigcup\left(\R^n\setminus\{0\}\right)\longrightarrow \R^n,\quad x\mapsto {x}/{|x|^2}.
$$
We define the framing at $*$ to be the natural identification  $\R^n\cong T_0\R^n=T_{s(\infty)}\R^n\cong T_\infty S^n$.
The left action $c\circ_i X$, where $c\in \F_n(\ell)$,   replaces the point $*$ by an infinitesimal  configuration $s_i(c)$, where
$$
s_i\colon\F_n\to\F_n
$$
is the inversion with the center $i$-th point. The replacement is done accordingly to the framing at infinity that we described above.

The reader might have an impression that this construction is more difficult than it actually is. It might appear difficult only because  when points escape to infinity we describe how  the situation looks like from the point of view of $\infty\in S^n$. But if we always  keep the global picture in mind  from the point of view of the observer in $\R^n$ than we can see that there is no any twist in the framing and everything remains as flat as $\R^n$ is.

To be precise there exist two ways to describe $C_*[k,S^n]$ and its
strata. The first one is the usual one, see~\cite{SinhaCompact}, which we
call {\it spherical}, and which was used above to describe the
infinitesimal action of $\F_n$ on $C_*[\bullet,S^n]$. In the second
description, that we call {\it flat}, it is much easier to see that
$C_*[\bullet,S^n]$ is naturally an infinitesimal bimodule over $\F_n$. The
difference is that in the spherical model we look how points approach
$*=\infty$ in $S^n$, while in the flat model we look how points escaping
to infinity are located one with respect to the other in $\R^n$. As
example, consider the situation when points 3, 4, and 5, remain fixed in
$\R^n$, and points 1 and 2 escape to infinity. The corresponding stratum
in $C_*[5,S^n]$ is the product $C(3,\R^n)\times C(3,\R^n)/G$, where $G$ is
the group of translations and positive rescalings. In the flat model, the
first factor describes the location of 3, 4, 5 in $\R^n$; the second
factor describes the relative location of 1, 2, and conglomeration
$x=\{3,4,5\}$. We can represent a point in such stratum as follows:



\begin{center}
\psfrag{Rm}[0][0][1][0]{$\R^n$}
\includegraphics[width=6cm]{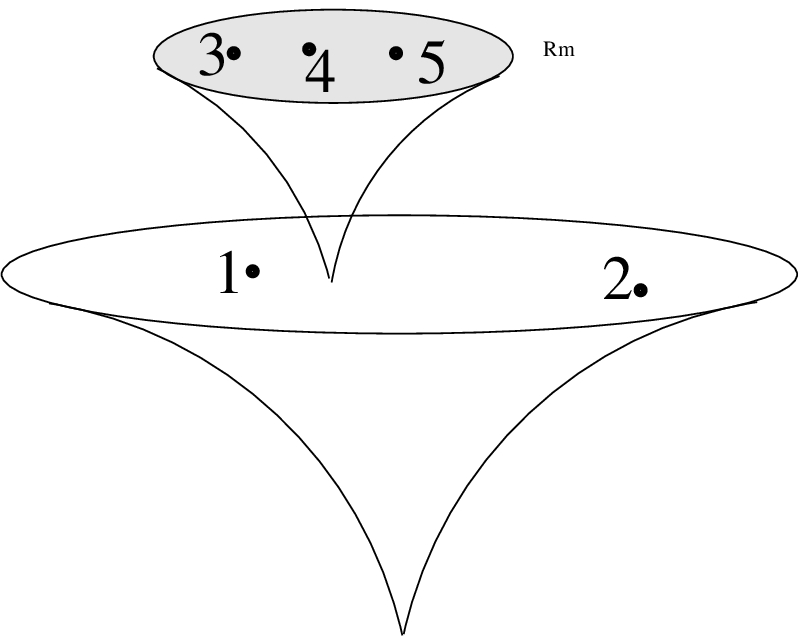}
\end{center}
The upper disc describes the actual world $\R^n$ (it is not quotiented out
by $G$); the lower disc describes how points escape to infinity. In
particular we see that~2 escapes to infinity approximately~5 times faster
than~1. However, from the perspective of the point at infinity the picture
is different:


\begin{center}
\includegraphics[width=3cm]{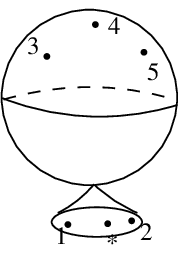}
\end{center}

Since 2 escapes to $\infty$ faster, it is closer to $*$ than~1. The
configuration of points 1, 2, $*$ at infinity is obtained from the \lq\lq flat
configuration" of 1, 2, $x=\{3,4,5\}$, by taking inversion with center $x$.

As a more general example, consider the stratum of $C_*[8,S^n]$ encoded by
the tree:


\begin{center}
\includegraphics[width=8cm]{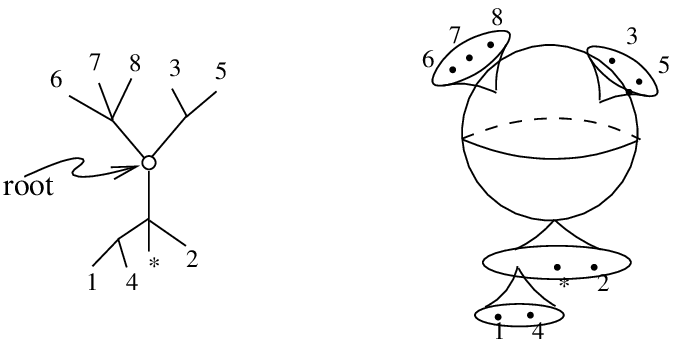}
\end{center}

For the configurations in this stratum, points 6, 7, 8 collide together;
similarly 3 and 5 collide; points 1, 4, and 2 escape to infinity, but
while doing so 1 stays close to 4.


\begin{center}
\psfrag{Rm}[0][0][1][0]{$\R^n$}
\includegraphics[width=10cm]{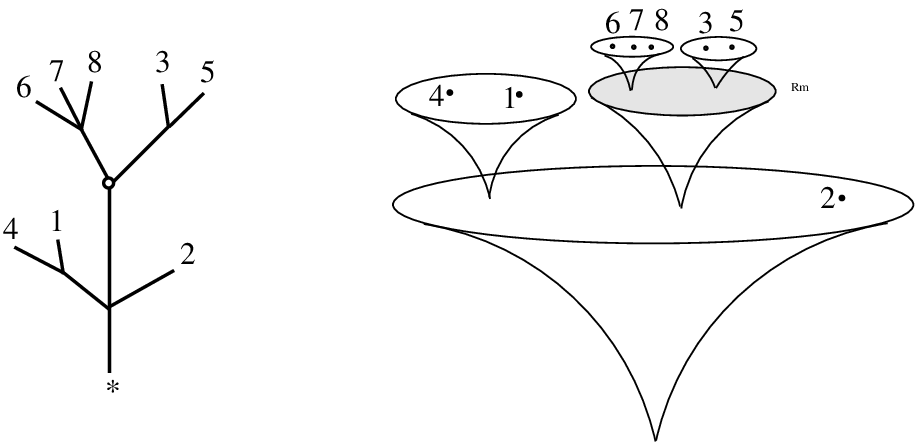}
\end{center}

The figure above describes the corresponding limiting configuration in the
flat model, in which we look how the points are located one with respect
to the other, rather than how they are located with respect to $\infty$.
The shaded disc in the above figure is the actual world $\R^n$. Both flat
and spherical models parametrize their strata as products
$$
C(|r|,\R^n)\times\prod_{v\in V(T)}C(|v|,\R^n)/G,
$$
where $T$ is a tree encoding the stratum, $r$ is its root, $V(T)$ is its
set of non-root-non-leaf vertices, $|v|$ is the valence of $v$ minus~1. To
pass from the spherical parametrization of a stratum encoded by a tree $T$ to the flat one, one needs to take inversion of the factors $C(|v|,\R^n)/G$ that correspond to the vertices $v$ lying on
the path between $*$ and the root $r$ in $T$. For all the other factors
the map, that gives correspondence, is identity.

Mention that the flat description of $C_*[k,S^n]$ is alluded in the
Bott-Taubes integration when one considers \lq\lq strata at infinity", as for example in~\cite{Catt}.

To recall both  $C_*[k,S^n]$ and $\F_n(k)$ are manifolds with corners whose interiors are  respectively $C(k,\R^n)$ and $C(k,\R^n)/G$.

\begin{lemma}\label{l_inf_bimod_equiv}
The projection $C(k,\R^n)\to C(k,\R^n)/G$, where $G$ is the group of translations and positive rescalings, induces a continuous map $C_*[k,S^n]\to \F_n(k)$, $k\geq 0$, which defines an equivalence of infinitesimal $\F_n$ bimodules
$C_*[\bullet,S^n]\to F_n(\bullet)$.
\end{lemma}

We skip the proof of this lemma.
The most difficult part of the proof is probably  checking that the induced map is a morphism of infinitesimal bimodules, which is however straightforward from the flat description of $C_*[k,S^n]$.

\begin{theorem}\label{t_model_long2}
For all $k\geq 0$ and all $n\geq m$ one has an equivalence
$$
T_k\Ebarmn\simeq \underset{\F_m}{\operatorname{hIbimod}}{}_{\le k}(C_*[\bullet,S^m],
C_*[\bullet,S^n]).
$$
\end{theorem}

This theorem is a consequence of Proposition~\ref{p_infmod_cof} and Theorem~\ref{t_model_long3} below.

Notice that we are in a similar situation as in Section~\ref{s2}: one can see that $C_*[\bullet,S^m]$ is cofibrant as an infinitesimal bimodule over the reduced Fulton-MacPherson operad $\bar\F_m$. Thus one only needs to correct it a little bit in order to make the (right) action of $\F_m(0)$ to be free. Define $\widetilde C_*[\bullet,S^m]$ as the sequence of spaces
$$
\widetilde C_*[k,S^m]=\left.\left(\coprod_{\ell\geq 0}\widetilde C_*[k,S^m]\times\left.\left((S^m)^{\times \ell}\times [0,1]^\ell\right)\right/\Sigma_\ell\right)\right/\sim
$$
of hairy configurations. The equivalence relations are~\eqref{eq_hairs_sigma}, \eqref{eq_hairs_0}, \eqref{eq_hairs_collide}, plus in addition~\eqref{eq_hairs_collide} must hold whenever $y_i=\infty$. In words the last condition says that a hair must disappear whenever it approaches infinity.

Define filtration
$$
\widetilde C_*^{(0)}[\bullet,S^m]\subset \widetilde C_*^{(1)}[\bullet,S^m]\subset \widetilde C_*^{(2)}[\bullet,S^m]\subset\ldots
$$
similar to~\eqref{eq_filtration2} by the number of geometrically distinct points different from~$\infty$.

\begin{proposition}\label{p_infmod_cof}
The natural projection
$$
\widetilde C_*^{(k)}[\bullet\bigr|_{\leq k},S^m]\to C_*[\bullet\bigr|_{\leq k}, S^m]
$$
is a cofibrant replacement of $k$-truncated infinitesimal bimodules over $\F_m$.
\end{proposition}

The proof is similar to~\ref{p_cofibr}.

\begin{theorem}\label{t_model_long3}
For $n\geq m$ and any $k\geq 0$ one has
$$
 T_k\Ebarmn\simeq \underset{\F_m}{\operatorname{Ibimod}}{}_{\le k}(\widetilde C_*^{(k)}[\bullet,S^m],
C_*[\bullet,S^n]).
$$
\end{theorem}

 In the above $\operatorname{Ibimod}{}_{\leq k}$ denote the space of (non-derived) morphisms of truncated infinitesimal bimodules. The main idea  is that for any $U\in\widetilde\calO(\R^m)$ one can similarly define infinitesimal bimodules
$\widetilde C_*[\bullet,U]$ and then one can prove statements similar to  Proposition~\ref{p_hom_sheaf} and Proposition~\ref{p_eval} which imply the result.

\subsection*{Acknowledgement}
The author is greatful to G.~Arone, P.~Boavida de Brito, B.~Fresse, J.~Hughes, M.~Kontsevich, P.~Lambrechts, P.~Salvatore, and B.~Vallette for discussions and communication.

\end{document}